\documentclass[12pt,a4paper]{article}
\usepackage{bm}
\usepackage{url}
\usepackage{chngcntr}
\usepackage{amsmath}
\usepackage{amssymb}
\usepackage{natbib}
\usepackage[ruled,vlined]{algorithm2e}
\usepackage{amsthm}
\usepackage{float}
\usepackage{setspace}
\usepackage{graphicx}
\usepackage[nottoc,notlot,notlof]{tocbibind}
\usepackage{authblk}
\usepackage[top=2cm, bottom=2cm, left=2cm, right=2cm]{geometry}
\usepackage{mathtools}
\usepackage{array}
\usepackage{lmodern,blindtext}
\usepackage{float}
\floatstyle{plaintop}
\restylefloat{table}
\usepackage{hyperref,url}
\hypersetup{
	colorlinks=true,
  linkcolor=red,          
  citecolor=blue,         
  filecolor=magenta,      
  urlcolor=cyan           
}

\newcounter{relctr} 

\newcommand\labelrel[2]{%
  \begingroup
    \refstepcounter{relctr}%
    \stackrel{\textnormal{(\alph{relctr})}}{\mathstrut{#1}}%
    \originallabel{#2}%
  \endgroup
}
\AtBeginDocument{\let\originallabel\label}
\input xy

\xyoption{all}

\newtheorem{prop}{Proposition}[section]
\newtheorem{lemma}[prop]{Lemma}
\newtheorem{theorem}[prop]{Theorem}

\newtheorem{cor}[prop]{Corollary}
\theoremstyle{definition}
\newtheorem{definition}[prop]{Definition}
\newtheorem{remark}[prop]{Remark}

\newcommand{\N}{\mathbb{N}}
\newcommand{\R}{\mathbb{R}}
\newcommand{\Z}{\mathbb{Z}}

\counterwithout{equation}{section}
\allowdisplaybreaks

\begin{document}
\pagenumbering{arabic}
\setcounter{page}{1}

\title{The expected degree of noninvertibility of compositions of functions and a related combinatorial identity}

\author{Sela Fried\thanks{A postdoctoral fellow in the Department of Computer Science at the Ben-Gurion University of the Negev. 
.}}
\date{} 
\maketitle

\begin{abstract}
Recently, \citet{defant2020quantifying} defined the degree of noninvertibility of a function $f\colon X\to Y$ between two finite nonempty sets by $\deg(f)=\frac{1}{|X|}\sum_{x\in X}|f^{-1}(f(x))|$. We obtain an exact formula for the expected degree of noninvertibility of the composition of $t$ functions for every $t\in \N$. An equivalent formulation for the definition of the degree of noninvertibility is then the starting point for a generalization yielding a seemingly new combinatorial identity involving the Stirling transform of the signed Stirling numbers of the first kind.
\end{abstract} 

\section{Introduction}

Recently, \citet{defant2020quantifying} defined the degree of noninvertibility of a function $f\colon X\to Y$ between two finite nonempty sets by $$\deg(f)=\frac{1}{|X|}\sum_{x\in X}|f^{-1}(f(x))|$$ as a measure of how far $f$ is from being injective. Interested mainly in endofunctions (also called dynamical systems within the field of dynamical algebraic combinatorics), that is, functions $f\colon X\to X$, they then computed the degrees of noninvertibility of several specific functions and studied the connection between the degrees of noninvertibility of functions and those of their iterates.

\section{Main results}

The purpose of this work is to continue the research of this newly introduced notion of degree of noninvertibility by addressing the following question: Let $t\in\N$. What is the expected degree of noninvertibility of the composition of $t$ functions? We prove

\begin{theorem}\label{hsg}
Let $t\in\N$ and let $X_1,\ldots,X_{t+1}$ be finite nonempty sets of sizes $n_1,\ldots,n_{t+1}$, respectively. Denote $$\mathcal{D}(X_1,\ldots,X_{t+1}) = \frac{1}{\prod_{s=1}^{t-1}n_s^{n_{s+1}}}\sum_{\substack{f_s\colon X_s\to X_{s+1}\\1\leq s\leq t}}\deg(f_t\circ\cdots\circ f_1).$$ Then $$\mathcal{D}(X_1,\ldots,X_{t+1}) =\frac{\prod_{s=1}^{t+1}n_s-\prod_{s=1}^{t+1}(n_s-1)}{\prod_{s=2}^{t+1}n_s}.$$
\end{theorem}

In the special case that all the sets in Theorem \ref{hsg} are equal, we obtain

\begin{cor}\label{mbs}
Let $t\in\N$ and let $X$ be a finite nonempty set of size $n$. Then $$\mathcal{D}(\overbrace{X,\ldots,X}^{t+1\textnormal{ times}}) =\frac{n^{t+1}-(n-1)^{t+1}}{n^t}.$$
\end{cor}

\begin{remark}\label{g65}
In the notation of Corollary \ref{mbs}, we have
\begin{equation}\label{uuu}n^{t+1}-(n-1)^{t+1}=\sum_{s=0}^{t}(-1)^{s}\binom{t+1}{s+1}n^{t-s}.\end{equation} Thus, the coefficients of the different powers of $n$ correspond to the beheaded rows of Pascal's triangle with alternating signs (related to \href{https://oeis.org/A074909}{A074909} in the OEIS). For example, for $t=1,2,3$, the right-hand side of (\ref{uuu}) takes the form
\begin{align}
    & 2n-1,\nonumber\\
    & 3n^2-3n+1,\nonumber\\
    & 4n^3-6n^2+4n-1.\nonumber
\end{align} It immediately follows that
$$\mathcal{D}(\overbrace{X,\ldots,X}^{t+1\textnormal{ times}}) \underset{n\to\infty}{\longrightarrow} t+1.$$
\end{remark}

We continue by strengthening \cite[Theorem 3.4]{defant2020quantifying} that states that if $X$ is a finite set of size $n$ then $$\deg(f\circ g)\leq \sqrt{n}\sqrt{\deg(f)}\deg(g),\;\; \forall f,g\colon X\to X.$$ We prove

\begin{theorem}\label{tha}
Let $X, Y$ and $Z$ be three finite nonempty sets and let $g\colon X\to Y$ and $f\colon Y\to Z$ be two functions. Then $$\deg(f\circ g)\leq \max_{z\in Z}\{|f^{-1}(z)|\}\deg(g).$$
\end{theorem}

That Theorem \ref{tha} is a strengthening of \cite[Theorem 3.4]{defant2020quantifying} in the case $Y=Z=X$ follows from Lemma \ref{l; 22} but the extent of this strengthening may be appreciated by comparing the order of the expectation of $\sqrt{n}\sqrt{\deg(f)}$ which is $\Theta(\sqrt{n})$ (cf.\ Remark \ref{g65}) with the one of $\max_{z\in X}\{|f^{-1}(z)|\}$ which is $\Theta\left(\frac{\log (n)}{\log(\log(n))}\right)$ (a result due to \cite{gonnet1981expected}. See also the references in \href{https://oeis.org/A208250}{A208250} in the OEIS).

\bigskip

It is easy to see that if $f\colon X\to Y$ is a function between two finite nonempty sets, then $$\deg(f)=\frac{1}{|X|}\sum_{y\in Y}|f^{-1}(y)|^2.$$ This formulation (which we shall use freely throughout this work) opens the door for a generalization: For a function $f\colon X\to Y$ and $q\in\N$ we define $$\deg(f, q)=\frac{1}{|X|}\sum_{y\in Y}|f^{-1}(y)|^q.$$ We prove the following theorem which makes use of the notations ${n\brace k}$ and ${ n\brack k}$ that denote the Stirling numbers of the second and of the first kind, respectively (e.g., \cite[pp. 243--253]{graham1989concrete}).

\begin{theorem}\label{dse}
Let $X$ and $Y$ be two finite nonempty sets of sizes $n$ and $m$, respectively and let $q\in\N$. Then 
$$\frac{1}{nm^n}\sum_{f\colon X\to Y}\deg(f,q)=\frac{1}{m^{q-1}}\sum_{k=1}^{q}{q \brace k}\left(\sum_{j=1}^{k}(-1)^{k-j}{k \brack j}n^{j-1}\right)m^{q-k}.$$
\end{theorem}

By taking $Y=X$ in Theorem \ref{dse} we obtain a seemingly new combinatorial identity involving the Stirling transform of the signed Stirling numbers of the first kind (cf.\ \cite{bernstein1995some} and \href{https://oeis.org/A118984}{A118984} in the OEIS)):

\begin{cor}\label{coj}
Let $n,q\in\N$. Then 
$$\sum_{\substack{k_{1},\dots,k_n\geq0\\
k_{1}+\cdots+k_n=n}}\binom{n}{k_{1},\dots,k_n}\left(\sum_{i=1}^nk_{i}^q\right)=n^{n-(q-2)}\sum_{k=0}^{q-1}(-1)^{k}\left(\sum_{j=1}^{q-s}{q \brace k+j}{k+j \brack j}\right)n^{q-k-1}.$$ In particular, $$\sum_{k=0}^{q-1}(-1)^{k}\left(\sum_{j=1}^{q-k}{q \brace k+j}{k+j \brack j}\right)=1.$$
\end{cor}

\section{The proofs}

\begin{definition}
Let $t\in\N$ and let $X_1,\ldots,X_{t+1}$ be  finite nonempty sets of sizes $n_1,\ldots,n_{t+1}$, respectively. Denote $$\mathcal{D}(X_1,\ldots,X_{t+1}) = \frac{1}{\prod_{s=1}^{t}n_{s+1}^{n_{s}}}\sum_{\substack{f_s\colon X_{s}\to X_{s+1}\\1\leq s\leq t}}\deg(f_t\circ\cdots\circ f_1).$$
\end{definition}

The proof of Theorem \ref{hsg} relies on the following lemma

\begin{lemma}\label{iy}
Let $t\in\N$ and let $X_1,\ldots,X_{t+1}$ be finite nonempty sets of sizes $n_1,\ldots,n_{t+1}$, respectively. It holds
{\scriptsize
\begin{align}&\mathcal{D}(X_1,\ldots,X_{t+1}) =\nonumber\\ &\frac{1}{n_1\prod_{s=1}^{t}n_{s+1}^{n_{s}}}\sum_{\substack{k_{t,1},\ldots,k_{t,n_{t+1}}\geq0\\k_{t,1}+\cdots+k_{t,n_{t+1}}=n_t}}\binom{n_t}{k_{t,1},\ldots,k_{t,n_{t+1}}}\sum_{\substack{k_{t-1,1},\ldots,k_{t-1,n_{t+1}}\geq0\\k_{t-1,1}+\cdots+k_{t-1,n_{t+1}}=n_{t-1}}}\binom{n_{t-1}}{k_{t-1,1},\ldots,k_{t-1,n_{t+1}}}k_{t,1}^{k_{t-1,1}}\cdots k_{t,n_{t+1}}^{k_{t-1,n_{t+1}}}\cdots\nonumber\\&\hspace{209pt}\sum_{\substack{k_{1,1},\ldots,k_{1,n_{t+1}}\geq0\\k_{1,1}+\cdots+k_{1,n_{t+1}}=n_1}}\binom{n_1}{k_{1,1},\ldots,k_{1,n_{t+1}}}k_{2,1}^{k_{1,1}}\cdots k_{2,n_{t+1}}^{k_{1,n_{t+1}}} \left(\sum_{i=1}^{n_{t+1}}k_{1,i}^{2}\right).\nonumber\end{align}}%
\end{lemma}

\begin{proof} 
Assume $X_{t+1}=\{x_1,\ldots,x_{n_{t+1}}\}$ and for every $1\leq s\leq t$ let $f_s\colon X_s\to X_{s+1}$. For every $1\leq i\leq n_{t+1}$ we define iteratively: $X_{t,i} = f_t^{-1}(x_i), k_{t,i}= |X_{t,i}|$ and for $1\leq s\leq t-1$: $X_{s,i} =  f_{s}^{-1}(X_{s+1,i}),  k_{s,i}=|X_{s,i}|$.
Then $$(f_t\circ\cdots \circ f_1)^{-1}(x_i) = X_{1,i}$$ and therefore $$|(f_t\circ\cdots \circ f_1)^{-1}(x_i)| = k_{1,i}$$ Now, for every $1 \leq s\leq t-1$ there are exactly $\binom{n_s}{k_{s,1},\ldots,k_{s,n_{t+1}}}k_{s+1,1}^{k_{s,1}}\cdots k_{s+1,n_{t+1}}^{k_{s,n_{t+1}}}$ functions $g\colon X_s\to X_{s+1}$ such that $|g^{-1}(X_{s+1,i})|=k_{s,i}, 1\leq i\leq n_{t+1}$.
\end{proof}

\paragraph{Proof of Theorem \ref{hsg}:} We proceed by induction on $t$. For the induction step we shall need the following identity from which we shall also deduce the case $t=1$: Let $m,n\in\N$ and let $k_1,\ldots,k_n\geq 0$. Denote $r=\sum_{i=1}^nk_i$. Then \begin{equation}\label{en}\sum_{\substack{l_{1},\ldots,l_n\geq0\\
l_{1}+\cdots+l_n=m}}\binom{m}{l_{1},\ldots,l_{n}}k_{1}^{l_{1}}\cdots k_{n}^{l_{n}}\left(\sum_{i=1}^{n}l_{i}^{2}\right)=m(m-1)r^{m-2}\sum_{i=1}^nk_i^2+mr^m.\end{equation} Indeed,

{\scriptsize
\begin{align}
&\sum_{\substack{l_{1},\ldots,l_n\geq0\\
l_{1}+\cdots+l_n=m}}\binom{m}{l_{1},\ldots,l_{n}}k_{1}^{l_{1}}\cdots k_{n}^{l_{n}}\left(\sum_{i=1}^{n}l_{i}^{2}\right)\nonumber\\=&\sum_{i=1}^n\sum_{\substack{l_{1},\ldots,l_{n}\geq0\\l_{1}+\cdots+l_m=n}}\binom{m}{l_{1},\ldots,l_{n}}l_{i}^{2}k_{1}^{l_{1}}\cdots k_{n}^{l_{n}}\nonumber\\=&\sum_{i=1}^n\Bigg(\sum_{\substack{l_{1}\geq0,\ldots,l_{i}=1,\ldots,l_{n}\geq0\\
l_{1}+\cdots+l_{n}=m}}\binom{m}{l_{1},\ldots,l_{n}}l_{i}^{2}k_{1}^{l_{1}}\cdots k_{n}^{l_{n}}+\sum_{\substack{l_{1}\geq0,\ldots, l_{i}\geq2,\ldots,l_{n}\geq0\\ l_{1}+\cdots+l_{n}=m}}\binom{m}{l_{1},\ldots,l_{n}}l_{i}^{2}k_{1}^{l_{1}}\cdots k_{n}^{l_{n}}\Bigg)\nonumber\\=&\sum_{i=1}^n\Bigg(mk_{i}(r-k_{i})^{m-1}+m\sum_{\substack{l_{1}\geq0,\ldots,l_{i}\geq2,\ldots,l_{n}\geq0\\
l_{1}+\cdots+l_{n}=m}}\binom{m-1}{l_{1},\ldots,l_i-1,\ldots,l_{n}}(l_{i}-1+1)k_{1}^{l_{1}}\cdots k_{n}^{l_{n}}\Bigg)\nonumber\\=&\sum_{i=1}^n\Bigg(mk_{i}(r-k_{i})^{m-1}+k_{i}^{2}m(m-1)\sum_{\substack{l_{1}\geq0,\ldots, l_{i}-2\geq 0,\ldots,l_{n}\geq0\\
l_{1}+\cdots+l_i-2+\cdots+l_{n}=m-2}}\binom{m-2}{l_{1},\ldots,l_i-2,\ldots,l_{n}}k_{1}^{l_{1}}\cdots k_{i}^{l_{i}-2}\cdots k_{n}^{l_{n}}+\nonumber \\ &\hspace{124pt} mk_i\sum_{\substack{l_{1}\geq0,\ldots, l_{i}-1\geq1,\ldots,l_{n}\geq0\\
l_{1}+\cdots+l_i-1+\cdots+l_{n}=m-1}}\binom{m-1}{l_{1},\ldots,l_i-1,\ldots,l_{n}}k_{1}^{l_{1}}\cdots k_{i}^{l_{i}-1}\cdots k_{n}^{l_{n}}\Bigg)\nonumber\\=&m\sum_{i=1}^nk_{i}\Bigg((r-k_{i})^{m-1}+k_{i}(m-1)r^{m-2}+\sum_{\substack{l_{1}\geq0,\ldots, l_{i}\geq1,\ldots,l_{n}\geq0\\
l_{1}+\cdots+l_{n}=m-1}}\binom{m-1}{l_{1},\ldots,l_i,\ldots,l_{n}}k_{1}^{l_{1}}\cdots k_{i}^{l_{i}}\cdots k_{n}^{l_{n}}\Bigg)	\nonumber\\=&m\sum_{i=1}^nk_i\Bigg((r-k_i)^{m-1}+k_i(m-1)r^{m-2}+\sum_{\substack{l_{1},\ldots,l_n\geq0\\
l_{1}+\cdots+l_n=m-1}}\binom{m-1}{l_{1},\ldots,l_n}k_{1}^{l_{1}}\cdots k_{n}^{l_n}-\sum_{\substack{l_{1}\geq0,\ldots, l_{i}=0,\ldots,l_{n}\geq0\\
l_{1}+\cdots+l_{n}=m-1}}\binom{m-1}{l_{1},\ldots,l_n}k_{1}^{l_{1}}\cdots k_n^{l_n}\Bigg)\nonumber\\
=&m\sum_{i=1}^nk_i\left((r-k_i)^{m-1}+k_i(m-1)r^{m-2}+r^{m-1}-(r-k_i)^{m-1}\right) \nonumber \\ = &m(m-1)r^{m-2}\sum_{i=1}^nk_i^2+mr^m.\nonumber
\end{align}}%
Let $t=1$. It holds  \begin{align}\mathcal{D}(X_1,X_2) \overset{\textnormal{Lemma } \ref{iy}}{=}& \frac{1}{n_1n_2^{n_1}}\sum_{\substack{k_{1,1},\ldots,k_{1,n_2}\geq1\\
k_{1,1}+\cdots+k_{1,n_2}=n_1}}\binom{n_1}{k_{1,1},\ldots,k_{1,n_2}}\left(\sum_{i=1}^{n_2}k_{1,i}^{2}\right)\nonumber\\\overset{(\ref{en})}{=}&\frac{n_1(n_1-1)n_2^{n_1-2}n_2+n_1n_2^{n_1}}{n_1n_2^{n_1}}\nonumber\\=&\frac{n_1+n_2-1}{n_2}\nonumber\nonumber\\=&\frac{n_1n_2-(n_1-1)(n_2-1)}{n_2}\nonumber.\end{align} 
Suppose now that the claim holds for $t\in\N$. We prove that it holds for $t+1$:
{\scriptsize\begin{align}
&\mathcal{D}(X_1,\ldots,X_{t+2})\nonumber \\ \overset{\textnormal{Lemma } \ref{iy}}{=}& \frac{1}{n_1\prod_{s=1}^{t+1}n_{s+1}^{n_{s}}}\sum_{\substack{k_{t+1,1},\ldots,k_{t+1,n_{t+2}}\geq0\\
k_{t+1,1}+\cdots+k_{t+1,n_{t+2}}=n_{t+1}}}\binom{n_{t+1}}{k_{t+1,1},\ldots,k_{t+1,n_{t+2}}}\sum_{\substack{k_{t,1},\ldots,k_{t,n_{t+2}}\geq0\\k_{t,1}+\cdots+k_{t,n_{t+2}}=n_{t}}}\binom{n_{t}}{k_{t,1},\ldots,k_{t+1,n_{t+2}}}k_{t+1,1}^{k_{t,1}}\cdots k_{t+1,n_{t+2}}^{k_{t,n_{t+2}}}\cdots\nonumber\\&\hspace{200pt}\sum_{\substack{k_{1,1},\ldots,k_{1,n_{t+2}}\geq0\\k_{1,1}+\cdots+k_{1,n_{t+2}}=n_1}}\binom{n_1}{k_{1,1},\ldots,k_{1,n_{t+2}}}k_{2,1}^{k_{1,1}}\cdots k_{2,n_{t+2}}^{k_{1,n_{t+2}}} \left(\sum_{i=1}^{n_{t+2}}k_{1,i}^{2}\right)\nonumber\\\overset{(\ref{en})}{=}& \frac{1}{n_1\prod_{s=1}^{t+1}n_{s+1}^{n_{s}}}\sum_{\substack{k_{t+1,1},\ldots,k_{t+1,n_{t+2}}\geq0\\
k_{t+1,1}+\cdots+k_{t+1,n_{t+2}}=n_{t+1}}}\binom{n_{t+1}}{k_{t+1,1},\ldots,k_{t+1,n_{t+2}}}\sum_{\substack{k_{t,1},\ldots,k_{t,n_{t+2}}\geq0\\k_{t,1}+\cdots+k_{t,n_{t+2}}=n_{t}}}\binom{n_{t}}{k_{t,1},\ldots,k_{t+1,n_{t+2}}}k_{t+1,1}^{k_{t,1}}\cdots k_{t+1,n_{t+2}}^{k_{t,n_{t+2}}}\cdots\nonumber\\&\hspace{150pt}\sum_{\substack{k_{2,1},\ldots,k_{2,n_{t+2}}\geq0\\k_{2,1}+\cdots+k_{2,n_{t+2}}=n_2}}\binom{n_2}{k_{2,1},\ldots,k_{2,n_{t+2}}}k_{3,1}^{k_{2,1}}\cdots k_{3,n_{t+2}}^{k_{2,n_{t+2}}} \left(n_1(n_1-1)n_2^{n_1-2}\sum_{i=1}^{n_{t+2}}k_{2,i}^2+n_1n_2^{n_1}\right)\nonumber\\=&\frac{n_1n_2^{n_1}\prod_{s=2}^{t+1}n_{s+1}^{n_{s}}}{n_1\prod_{s=1}^{t+1}n_{s+1}^{n_{s}}}+\nonumber\\&\frac{n_1-1}{n_{2}^2\prod_{s=2}^{t+1}n_{s+1}^{n_{s}}}\sum_{\substack{k_{t+1,1},\ldots,k_{t+1,n_{t+2}}\geq0\\
k_{t+1,1}+\cdots+k_{t+1,n_{t+2}}=n_{t+1}}}\binom{n_{t+1}}{k_{t+1,1},\ldots,k_{t+1,n_{t+2}}}\sum_{\substack{k_{t,1},\ldots,k_{t,n_{t+2}}\geq0\\k_{t,1}+\cdots+k_{t,n_{t+2}}=n_{t}}}\binom{n_{t}}{k_{t,1},\ldots,k_{t+1,n_{t+2}}}k_{t+1,1}^{k_{t,1}}\cdots k_{t+1,n_{t+2}}^{k_{t,n_{t+2}}}\cdots\nonumber\\&\hspace{150pt}\sum_{\substack{k_{2,1},\ldots,k_{2,n_{t+2}}\geq0\\k_{2,1}+\cdots+k_{2,n_{t+2}}=n_2}}\binom{n_2}{k_{2,1},\ldots,k_{2,n_{t+2}}}k_{3,1}^{k_{2,1}}\cdots k_{3,n_{t+2}}^{k_{2,n_{t+2}}} \left(\sum_{i=1}^{n_{t+2}}k_{2,i}^2\right)\nonumber\\=&1+(n_1-1)\frac{\prod_{s=2}^{t+2}n_s-\prod_{s=2}^{t+2}(n_s-1)}{n_2\prod_{s=3}^{t+2}n_s}\nonumber\\=&\frac{\prod_{s=1}^{t+2}n_{s}-\prod_{s=1}^{t+2}(n_{s}-1)}{\prod_{s=2}^{t+2}n_{s}}.\nonumber
\end{align}}%
\qed

Before we prove Theorem \ref{tha}, let us complement \cite[Lemma 1.2]{defant2020quantifying}:

\begin{lemma}\label{l; 22}
Let $f\colon X\to Y$. Then $$\max_{y\in Y}\{|f^{-1}(y)|\}\leq \sqrt{n}\sqrt{\deg(f)}.$$
\end{lemma}

\begin{proof}
Applying the inequality $$||x||_\infty \leq ||x||_2, \;\;\forall x\in\R^m$$ on the vector in $\R^m$ whose entries correspond to the sizes of the preimages under $f$ of all the elements of $Y$ we obtain
$$\max_{y\in Y}\{|f^{-1}(y)|\}\leq \sqrt{\sum_{y\in Y}|f^{-1}(y)|^2} =\sqrt{n}\sqrt{\frac{1}{n}\sum_{y\in Y}|f^{-1}(y)|^2}= \sqrt{n}\sqrt{\deg(f)}.$$
\end{proof}

Only a small modification of the proof of  \cite[Theorem 3.4]{defant2020quantifying} is necessary to prove Theorem \ref{tha}. We give the full proof for completeness:

\paragraph{Proof of Theorem \ref{tha}:}
Let $z_1,\ldots,z_r\in Z$ be such that $f(g(X)) = \{z_1,\ldots,z_r\}$. For every $1\leq i\leq r$ denote $k_i = |f^{-1}(z_i)|$ and let $y_{i1},\ldots, y_{ik_i}\in Y$ be such that $f^{-1}(z_i)=\{y_{i1},\ldots, y_{ik_i}\}$. Furthermore, for $1\leq i\leq r$ and $1\leq j\leq k_i$ denote $l_{ij}=|g^{-1}(y_{ij})|$. It holds \begin{align}
    \deg(f\circ g)=&\frac{1}{n}\sum_{i=1}^r |g^{-1}(f^{-1}(z_i))|^2\nonumber\\=&\frac{1}{n}\sum_{i=1}^r \left(\sum_{j=1}^{k_i}l_{ij}\right)^2\nonumber\\\labelrel\leq{myeq:equality99}&\frac{1}{n}\sum_{i=1}^rk_i \sum_{j=1}^{k_i}l_{ij}^2\nonumber\\\leq&\max_{z\in Z}\{|f^{-1}(z)|\}\overbrace{\frac{1}{n}\sum_{i=1}^r \sum_{j=1}^{k_i}l_{ij}^2}^{\labelrel={myeq:equality100}\deg(g)}\nonumber\\=&\max_{z\in Z}\{|f^{-1}(z)|\}\deg(g).\nonumber
\end{align} where in \eqref{myeq:equality99} we used the Cauchy-Schwarz inequality and \eqref{myeq:equality100} is due to the fact that $$g(X)\subseteq f^{-1}(f(g(X))) = \{y_{ij} \;|\; 1\leq i\leq r, 1\leq j\leq k_i\}.$$
\qed

\paragraph{Proof of Theorem \ref{dse}} We proceed by (complete) induction on $q$ and prove that for every $n\in\N$ and $q\in\N\cup\{0\}$ it holds $$\sum_{\substack{k_{1},\dots,k_m\geq0\\
k_{1}+\cdots+k_m=n}}\binom{n}{k_{1},\dots,k_m}\left(\sum_{i=1}^mk_{i}^{q}\right)=\begin{cases}nm^{n-(q-1)}\sum_{k=1}^{q}{q \brace k}\left(\sum_{j=1}^{k}(-1)^{k-j}{k \brack j}n^{j-1}\right)m^{q-k}& q>0\\ m^{n+1} &q=0.\end{cases}$$ The casses $q=0,1$ are trivial. Let $1<q\in\N$ and suppose that the assertion holds for every $m,n\in\N$ and every $0\leq r\leq q$. We shall prove that it holds for every $m,n\in\N$ and for $q+1$: 
\begin{align}
&\sum_{\substack{k_{1},\dots,k_m\geq0\\
k_{1}+\cdots+k_m=n}}\binom{n}{k_{1},\dots,k_m}\left(\sum_{i=1}^mk_{i}^{q+1}\right)\nonumber\\=&n\sum_{i=1}^m\sum_{\substack{k_{1}\geq 0,\dots,k_i\geq 1,\ldots,k_m\geq0\\k_{1}+\cdots+k_m=n}}\binom{n-1}{k_{1},\dots,k_{i}-1,\dots,k_m}k_{i}^q\nonumber\\=&n\sum_{i=1}^m\sum_{\substack{k_{1}\geq 0,\dots,k_i-1\geq 0,\ldots,k_m\geq0\\k_{1}+\cdots+k_m=n-1}}\binom{n-1}{k_{1},\dots,k_{i}-1,\dots,k_m}(k_{i}-1+1)^q\nonumber\\=&n\sum_{r=0}^q\binom{q}{r}\sum_{i=1}^m\sum_{\substack{k_{1}\geq 0,\dots,k_i-1\geq 0,\ldots,k_m\geq0\\k_{1}+\cdots+k_m=n-1}}\binom{n-1}{k_{1},\dots,k_{i}-1,\dots,k_m}(k_{i}-1)^r\nonumber\\=&n\sum_{r=0}^q\binom{q}{r}\sum_{\substack{k_{1},\ldots,k_m\geq0\\k_{1}+\cdots+k_m=n-1}}\binom{n-1}{k_{1},\dots,k_m}\left(\sum_{i=1}^mk_i^r\right)\nonumber\\\labelrel={myeq:equality}&n\sum_{r=1}^q\binom{q}{r} (n-1)m^{n-r}\sum_{k=1}^{r}{r\brace k} \left(\sum_{j=1}^{k}(-1)^{k-j}{ k\brack j}(n-1)^{j-1}\right)m^{r-k}+nm^n\nonumber\\=&nm^{n-t}\sum_{k=1}^{q}\overbrace{\sum_{r=k}^q\binom{q}{r} {r\brace k}}^{\labelrel={myeq:equality1} {q+1\brace k+1}} \left(\sum_{j=1}^{k}(-1)^{k-j}{ k\brack j}\overbrace{(n-1)^j}^{=\sum_{l=0}^j\binom{j}{l}(-1)^ln^{j-l}}\right)m^{q-k}+nm^n\nonumber\\=&nm^{n-q}\sum_{k=1}^{q} {q+1\brace k+1} \left(\overbrace{\sum_{j=1}^{k}\sum_{l=0}^j(-1)^{k+l-j}\binom{j}{l}{ k\brack j}n^{j-l}}^{\labelrel={myeq:equality2}\sum_{j=0}^{k}(-1)^{k-j}{k+1 \brack j+1}n^{j}}\right)m^{q-k}+nm^{n-q}m^q\nonumber\\=&nm^{n-q}\left(\sum_{k=1}^{q} {q+1\brace k+1} \left(\sum_{j=0}^{k}(-1)^{k-j}{k+1 \brack j+1}n^{j}\right)m^{q-k}+m^q\right)\nonumber\\=&nm^{n-q}\sum_{k=1}^{q+1} {q+1\brace k} \left(\sum_{j=1}^{k}(-1)^{k-j}{k \brack j}n^{j-1}\right)m^{q+1-k}\nonumber
\end{align} where in \eqref{myeq:equality} we used the induction hypothesis,  \eqref{myeq:equality1} is due to \cite[(6.15)]{graham1989concrete} and \eqref{myeq:equality2} follows after several algebraic manipulations together with \cite[(6.16)]{graham1989concrete}.
\qed

\paragraph{Proof of Corollary \ref{coj}:}
It is straightforward to derive the asserted identity from Theorem \ref{dse}. Now, recall (e.g., \cite{bernstein1995some}) that if $(a_l)_{l\in\N}$ is a sequence of real numbers then the Stirling transform $(b_l)_{l\in\N}$ of $(a_l)_{l\in\N}$ is given by $$b_l = \sum_{i=1}^l{l \brace i}a_i,\;\;\forall l\in\N.$$ Setting ${x \brack y}=0$ for every $x\in \N$ and $y\in \Z$ such that $y\leq 0$ we have 
\begin{align}
(-1)^k\left(\sum_{j=1}^{q-k}{q \brace k+j}{k+j \brack j}\right) =\sum_{j=1}^{q}{q \brace j}(-1)^{j-(j-k)}{j \brack j-k} \nonumber. 
\end{align}

\paragraph{Acknowledgements}
We are grateful to James Propp for suggesting us the consideration of compositions of functions.

\bibliography{bibliography}
\bibliographystyle{plainnat}

\end{document}